\documentclass[11pt,a4paper]{article}

\usepackage[english]{babel}
\usepackage{latexsym}
\usepackage{amsmath,amsfonts,amsthm,amssymb,amscd}
\usepackage[mathscr]{eucal}
\usepackage{enumerate}
\usepackage{graphicx}
\usepackage{fullpage}
\usepackage{url}
\usepackage{color}

\theoremstyle{definition}
\newcounter{Polya}

\newtheorem{theo}{Theorem}[section]
\newtheorem{pro}[theo]{Proposition}
\newtheorem{lemma}[theo]{Lemma}
\newtheorem{cor}[theo]{Corollary}
\newtheorem{defi}[theo]{Definition}
\newtheorem{obs}[theo]{Observation}
\newtheorem{rema}[theo]{Remark}
\newtheorem{prp}[theo]{Properties}
\newtheorem{exam}[theo]{Example}
\newtheorem{prps}[theo]{Propiedades}
\newtheorem{obss}[theo]{Observations}
\newtheorem{ejes}[theo]{Ejemplos}
\newtheorem{nota}[theo]{Notation}

\swapnumbers

\newenvironment{prop}{\smallskip\begin{pro}}{\end{pro}\smallskip}

\newenvironment{coro}{\smallskip\begin{cor}}{\end{cor}\smallskip}
\newenvironment{den}{\smallskip\begin{defi}}{\end{defi}\smallskip}

\def\oo{\infty}
\def\ri{\rightarrow}
\def\a{\alpha}
\def\b{\beta}

\def\ep{\varepsilon}
\def\o{\omega}

\def\ga{\gamma}

\def\ro{\rho}

\def\bmu{\boldsymbol{\mu}}
\def\bl{\boldsymbol{\ell}}

\def\Re{\mathop{\rm Re}\nolimits}

\def\N{\mathbb{N}}

\def\R{\mathbb{R}}
\def\C{\mathbb{C}}      

\def\M{\mathbb{M}}
\def\H{\mathbb{H}}
\def\bM{\mathbb{M}}
\def\L{\mathbb{L}} 
\def\A{\mathbb{A}} 
\def\m{\boldsymbol{m}}
\def\h{\boldsymbol{h}}
\def\bm{\boldsymbol{m}}
\def\l{\boldsymbol{l}}

\def\no{\nonumber}

\begin{document}

\title{Strongly regular sequences and proximate orders}

\author{Javier Jim\'enez-Garrido and Javier Sanz}
\date{\today}

\maketitle

\abstract{Summability methods for ultraholomorphic classes in sectors, defined in terms of a strongly regular sequence $(M_p)_{p\in\N_0}$, have been put forward by A. Lastra, S. Malek and the second author~\cite{lastramaleksanz3}.
We study several open questions related to the existence of kernels of summability constructed by means of analytic proximate orders.
In particular, we give a simple condition that allows us to associate a proximate order with a strongly regular sequence. Under this assumption, and through the characterization of strongly regular sequences in terms of so-called regular variation, we show that the
growth index $\gamma(\M)$ defined by V.Thilliez~\cite{thilliez} and the order of quasianalyticity $\o(\M)$ introduced by the second author~\cite{SanzFlat} are the same.
}

\section{Introduction}

The study of the existence and meaning of formal power series solutions to differential equations has a long history, going back at least to the works of L. Euler in the 18th century. Although these solutions are frequently divergent, under fairly general conditions the rate of growth of their coefficients is not arbitrary. Indeed, a remarkable result of E. Maillet~\cite{Maillet} in 1903 states that any solution $\hat{f}=\sum_{p\ge 0}a_pz^p$ for an analytic differential equation will be of some Gevrey order, that is,
there exist $C,A,k>0$ such that $|a_p|\le CA^p(p!)^{1/k}$ for every $p\ge 0$.

These series turn out to be Gevrey asymptotic representations of actual solutions defined in suitable domains, and there is the possibility of reconstructing such analytic solutions from the formal ones by a process known as multisummability (in a sense, an iteration of a finite number of elementary $k-$summability procedures), developed in the 1980's by J.-P. Ramis, J. \'Ecalle, W. Balser et al.
This technique has been proven to apply successfully to a plethora of situations concerning the study of formal power series solutions at a singular point for linear and nonlinear (systems of) meromorphic ordinary differential equations in the complex domain (see, to cite but a few, the works~\cite{Balserutx,BalserBraaksmaRamisSibuya,Braaksma,MartinetRamis2,RamisSibuya}), for partial differential equations (for example, \cite{Balsermulti,BalserMiyake,Hibino,Malek3,ma2,Ouchi,taya}), as well as for singular perturbation problems (see~\cite{BalserMozo,CanalisMozoSchafke,lastramaleksanz2}, among others).

However, it is known that non-multisummable (in the previous sense) formal power series solutions may appear for different kinds of equations. G. K. Immink in~\cite{Immink,Immink2} has considered difference equations with formal power series solutions whose coefficients grow at a precise, intermediate rate between those of Gevrey classes, called $1^+$ level. She obtained reconstruction results for actual solutions by the consideration of specific kernels and integral transforms, very well suited for the problem studied.
Recently, S. Malek~\cite{Malek4} has studied some singularly perturbed small step size difference-differential nonlinear equations whose formal solutions with respect to the perturbation parameter can be decomposed as sums of two formal series, one with Gevrey order 1, the other of $1^+$ level, a phenomenon already observed for difference equations~\cite{BraaksmaFaberImmink}. In a different context, V. Thilliez~\cite{Thilliez2} has proven some stability results for algebraic equations whose coefficients belong to a general ultraholomorphic class defined by means of a so-called strongly regular sequence (comprising, but not limiting to, Gevrey classes), stating that the solutions will remain in the corresponding class. All these examples made it interesting for us to provide the tools for a general, common treatment of summability in ultraholomorphic classes in sectors, extending the powerful theory of $k-$summability.
The task, achieved by A. Lastra, S. Malek and the second author~\cite{lastramaleksanz3}, consisted in the construction of pairs of kernel functions with suitable asymptotic and growth properties, in terms of which to define formal and analytic Laplace- and Borel-like transforms which allow one to construct the sum of a summable formal power series in a direction. The main inspiration came from the theory of moment summability methods developed by W.~Balser in~\cite[Section\ 5.5]{Balserutx}, which had already found its application to the analysis of formal power series solutions of different classes of partial differential equations
(for example, by S. Malek~\cite{Malek2} and by S. Michalik~\cite{Michalik2}), and also for so-called moment-partial differential equations, introduced by W. Balser and Y. Yoshino~\cite{BalserYoshino} and subsequently studied by S. Michalik~\cite{Michalik3,Michalik4}.
Our technique has been applied in~\cite{lastramaleksanz3} to the study of the summability properties of some formal solutions for moment partial differential equations, generalizing the work in~\cite{Michalik3}, and in~\cite{lastramaleksanz4} to the asymptotic study of the solutions of a class of singularly perturbed partial differential equations in whose coefficients there appear sums of formal power series in this generalized sense. However, some questions remained unsolved in the construction of such generalized summability methods, and the present paper aims at providing some answers, as we proceed to describe.

The Carleman ultraholomorphic classes $\tilde{\mathcal{A}}_{\M}(G)$ we consider are those consisting of holomorphic functions $f$ admitting an asymptotic expansion $\hat f=\sum_{p\ge 0}a_pz^p/p!$ in a sectorial region $G$ with remainders suitably bounded in terms of a sequence $\M=(M_p)_{p\in\N_0}$ of positive real numbers (we write $f\sim_{\M}\hat f$ and $(a_p)_{p\in\N_0}\in\Lambda_{\M}$). The map sending $f$ to $(a_p)_{p\in\N_0}$ is the asymptotic Borel map $\tilde{\mathcal{B}}$.
See Subsection~\ref{subsectCarlemanclasses} for the precise definitions of all these classes and concepts.
In order to obtain good properties for these classes, the sequence $\M$ is usually subject to some standard conditions; in particular, we will mainly consider strongly regular sequences as defined by V. Thilliez~\cite{thilliez}, see Subsection~\ref{stronregseq}.
The best known example is that of Gevrey classes, appearing when the sequence is chosen to be $\bM_{\a}=(p!^{\a})_{p\in\N_{0}}$, $\a>0$,
and for which we use the notations $\tilde{\mathcal{A}}_{\a}(G)$, $\Lambda_{\a}$, $f\sim_{\a}\hat{f}$ and so on, for simplicity.
Let us denote by $G_\ga$ a sectorial region bisected by the direction $d=0$ and with opening $\pi\ga$. It is well known that $\tilde{\mathcal{B}}:\tilde{\mathcal{A}}_{\a}(G_\ga)\to\Lambda_{\a}$ is surjective if, and only if, $\ga\le\a$ (Borel--Ritt--Gevrey theorem, see~\cite{Ramis1,Ramis2}),
while it is injective 
if, and only if, $\ga>\a$ (Watson's lemma, see for example~\cite[Prop.\ 11]{Balserutx}).
The second author introduced in~\cite{SanzFlat} a constant $\omega(\M)\in(0,\infty)$, measuring the rate of growth of any strongly regular sequence $\M$,
in terms of which Watson's Lemma and Borel-Ritt-Gevrey theorem can be generalized in the framework of Carleman ultraholomorphic classes, as long as the associated function $d_{\M}(t)=\log(M(t))/\log t$, where
\begin{equation*}
M(t):=\sup_{p\in\N_{0}}\log\big(\frac{t^p}{M_{p}}\big)
,\quad t>0
,
\end{equation*}
is a proximate order, a concept appearing in the classical theory of growth for holomorphic functions in sectorial regions, with contributions by E. Lindel\"of, G. Valiron, V. Bernstein, A. A. Gol'dberg, I. V. Ostrovskii (see \cite{Valiron42,Levin,GoldbergOstrowskii}) and, more recently, L. S. Maergoiz~\cite{Maergoiz}. Moreover, if this is the case flat functions are available in optimal sectors, and this is the crucial point for the success in putting forward a satisfactory summability theory. So, it becomes extremely important to characterize the fact that $d_{\M}$ is a proximate order in a simple way, and this is the main purpose of the present paper. The second author stated the following result in terms of the sequence of quotients $\bm=(m_p)_{p\in\N_0}$, where $m_p=M_{p+1}/M_p$.

\begin{prop}[\cite{SanzFlat},\ Prop.\ 4.9]\label{propcaracdderordenaprox0}
Let $\M$ be a strongly regular sequence. The following are equivalent:
\begin{itemize}
\item[(i)] $d_{\M}(t)$ is a proximate order,
\item[(ii)] $\displaystyle\lim_{p\to\infty}\frac{p}{M(m_{p})}=\frac{1}{\omega(\M)}$.
\end{itemize}
\end{prop}

\noindent The argument rested on the following statement.

\begin{theo}[\cite{SanzFlat},\ Th.\ 2.14]
Let $\M$ be strongly regular. Then,
\begin{equation}\label{equaordeMdet1}
\lim_{r\to\infty}d_{\M}(t)=\frac{1}{\omega(\M)}.
\end{equation}
\end{theo}

Regrettably, the proof of this last result was not correct, as it was indicated to us by G. Schindl in a private communication calling our attention to some miscalculation. Although the limit in \eqref{equaordeMdet1} does exist in every example we have studied, in general we may only guarantee that $\limsup_{r\to\infty}d_{\M}(t)=1/\omega(\M)$. However, a careful analysis has shown that not only Proposition~\ref{propcaracdderordenaprox0} holds, but $d_{\M}(t)$ being a proximate order is also equivalent to the verification of \eqref{equaordeMdet1}, and even to the fact that
$$
\lim_{p\ri \oo} \frac{\log (m_p)}{\log(p)}=\o(\M).
$$
This last condition is really easy to verify in concrete examples, and it turns out that all the interesting examples of strongly regular sequences appearing in the literature satisfy it. In the proof of these characterizations, a deep result of H.-J. Petzsche~\cite{Pet} on equivalent sequences (see Proposition~\ref{petzche11} in this paper) and a very recent criterion by F. Moricz~\cite{Moricz} for the convergence of a sequence summable by Riesz means (here, Theorem~\ref{theoMoricz}) play a prominent role.

A different problem concerns the growth index $\gamma(\M)$, introduced by V. Thilliez~\cite{thilliez} for every strongly regular sequence. For the standard instances appearing in the literature, the values of the constants $\ga(\M)$ and $\omega(\M)$ agree, but in~\cite{SanzFlat} only the inequality $\ga(\M)\le \omega(\M)$ could be proved in general. In the last section we relate strong regularity to the property of regular variation of the sequence of quotients, what allows us to conclude that both indices coincide whenever $d_{\M}$ is a proximate order. We remark that the appearance of the concept of regular variation in this respect is not unnatural. On one hand, for every proximate order $\rho(t)$ the function $t^{\rho(t)}$ turns out to be regularly varying.
On the other hand, some arguments in the study of ultradifferentiable classes rely on this kind of ideas, see~\cite{bonetmeisemelikhov} as an example.

\section{Preliminaries}

\subsection{Notations}
We set $\N:=\{1,2,...\}$, $\N_{0}:=\N\cup\{0\}$.
$\mathcal{R}$ stands for the Riemann surface of the logarithm, and
$\C[[z]]$ is the space of formal power series in $z$ with complex coefficients.\par\noindent
For $\gamma>0$, we consider unbounded sectors
$$S_{\gamma}:=\{z\in\mathcal{R}:|\hbox{arg}(z)|<\frac{\gamma\,\pi}{2}\}$$
or, in general, bounded or unbounded sectors
$$S(d,\gamma,r):=\{z\in\mathcal{R}:|\hbox{arg}(z)-d|<\frac{\gamma\,\pi}{2},\ |z|<r\},\quad
S(d,\gamma):=\{z\in\mathcal{R}:|\hbox{arg}(z)-d|<\frac{\gamma\,\pi}{2}\}$$
with bisecting direction $d\in\R$, opening $\gamma \,\pi$ and (in the first case) radius $r\in(0,\infty)$.\par\noindent

A sectorial region $G(d,\gamma)$ with bisecting direction $d\in\R$ and opening $\gamma\,\pi$ will be a domain in $\mathcal{R}$ such that $G(d,\gamma)\subset S(d,\gamma)$, and for every $\beta\in(0,\gamma)$ there exists $\rho=\rho(\beta)>0$ with $S(d,\beta,\rho)\subset G(d,\gamma)$.
In particular, sectors are sectorial regions.\par\noindent

A sector $T$ is a bounded proper subsector of a sectorial region $G$ (denoted by $T\ll G$) whenever the radius of $T$ is finite and $\overline{T}\setminus\{0\}\subset G$.
\par\noindent

$\mathcal{H}(U)$ denotes the space of holomorphic functions in an open set $U\subset\mathcal{R}$.\par\noindent

\subsection{Strongly regular sequences}\label{stronregseq}

Most of the information in this subsection is taken from the works of
H. Komatsu~\cite{komatsu}, H.-J. Petzsche~\cite{Pet}, H.-J. Petzsche and D. Vogt~\cite{petvog} and V. Thilliez~\cite{thilliez},
which we refer to for further details and proofs. In what follows, $\M=(M_p)_{p\geq 0}$ always stands for
a sequence of positive real numbers, and we always assume that $M_0=1$. There is no general agreement in the literature on
the terminology describing the forthcoming properties. For the convenience of the reader, we have used the names given by
V.~Thilliez and the descriptive shortcuts employed by G. Schindl~\cite{Schindl}.

\begin{defi}  We say $\M$ is {\it strongly regular} if the following hold:
\begin{itemize}
\item[(i)]  $\M$ is \textit{logarithmically convex} (for short, (lc)), that is,
$$M_{p}^{2}\le M_{p-1}M_{p+1},\qquad p\in\N.
$$
\item[(ii)]  $\M$ is of \textit{moderate growth} (briefly, (mg)), i.e., there exists $A>0$ such that
$$M_{p+q}\le A^{p+q}M_{p}M_{q},\qquad p,q\in\N_0.$$
\item[(iii)]  $\bM$ satisfies the \textit{strong non-quasianalyticity condition} (for short, (snq)), that is, there exists $B>0$ such that
$$
\sum_{q\ge p}\frac{M_{q}}{(q+1)M_{q+1}}\le B\frac{M_{p}}{M_{p+1}},\qquad p\in\N_0.
$$
\end{itemize}
\end{defi}

\begin{defi} For a sequence $\M$ we define {\it the sequence of quotients} $\m=(m_p)_{p\in\N_0}$ by $$m_p:=\frac{M_{p+1}}{M_p} \qquad p\in \N_0.$$
\end{defi}

\begin{rema}\label{remaequivMm}
The properties (lc) and (snq) can be easily stated in terms of the sequence of quotients, and we will see in Proposition~\ref{propPropiedlcmg} that the same holds for (mg) as long as the given sequence $\M$ is (lc). Moreover,
observe that for every $p\in\N$ one has
\begin{equation}\label{eqMfromm}
M_p=\frac{M_p}{M_{p-1}}\frac{M_{p-1}}{M_{p-2}}\dots\frac{M_2}{M_{1}}\frac{M_1}{M_{0}}= m_{p-1}m_{p-2}\dots m_1m_0.
\end{equation}
So, one may recover the sequence $\M$ (with $M_0=1$) once $\bm$ is known, and hence the knowledge of one of the sequences amounts to that of the other.
Sequences of quotients of sequences $\M$, $\L$, etc. will be denoted by lowercase letters $\bm$, $\bl$ and so on. Whenever some statement
refers to a sequence denoted by a lowercase letter such as $\bm$, it will be understood that we are dealing with a sequence of quotients (of the sequence $\M$ given by \eqref{eqMfromm}).
\end{rema}

The following properties are easy consequences of the definitions, except for (ii.2), which is due to H.-J. Petzsche and D. Vogt~\cite{petvog}.

\begin{prop}\label{propPropiedlcmg}
Let $\M=(M_p)_{p\in\N_0}$ be a sequence. Then, we have:
\begin{itemize}
\item[(i)] $\M$ is (lc) if, and only if, $\bm$ is nondecreasing. If, moreover, $\M$ satisfies (snq) then $\bm$ tends to infinity.
\item[(ii)] Suppose from now on that $\M$ is (lc). Then
\begin{itemize}
\item[(ii.1)]  $(M_p^{1/p})_{p\in\N}$ is nondecreasing, and $M_p^{1/p}\leq m_{p-1}$ for every $p\in\N$.\par 
    Moreover, $\lim_{p\to\infty}m_p=\infty$ if, and only if, $\lim_{p\to\infty}M_p^{1/p}=\infty$.
\item[(ii.2)] The following statements are equivalent:
\begin{itemize}
  \item[(ii.2.a)]  $\M$ is (mg),
  \item[(ii.2.b)] $\sup_{p\in\N} \frac{m_p}{M^{1/p}_p} <\infty$,
  \item[(ii.2.c)]  $\sup_{p\in\N_0} \frac{m_{2p}}{m_{p}}<\infty$,
  \item[(ii.2.d)] $\sup_{p\in\N} \left(\frac{M_{2p}}{M_p^2}\right)^{1/p}<\infty$.
 \end{itemize}
\item[(ii.3)] If  $\M$ is (mg) and $A>0$ is the corresponding constant, then
\begin{equation}\label{eqmg_mppMp}
m_p^p\le A^{2p}M_p,\quad p\in\N_0.
\end{equation}
\end{itemize}
\end{itemize}
\end{prop}

In the next definitions and results we take into account the conventions adopted in Remark~\ref{remaequivMm}.

\begin{defi} Let $\M=(M_p)_{p\in\N_0}$ and $\L=(L_p)_{p\in\N_0}$ be sequences, we say that {\it $\M$ is equivalent to $\L$},
and we write $\M\approx\L$, if there exist $C,D>0$ such that
 $$D^p L_p \leq M_p \leq C^p L_p, \qquad\, p \in \N_0.$$
\end{defi}

\begin{defi} Let $\bm=(m_p)_{p\in\N_0}$ and $\bl=(\ell_p)_{p\in\N_0}$ be sequences,
we say that {\it $\bm$ is equivalent to $\bl$}, and we write $\bm\simeq\bl$, if there exist $c,d>0$ such that
 $$d  \ell_p \leq m_p \leq c \ell_p, \qquad p \in \N_0.$$
\end{defi}

The following statements are straightforward.

\begin{prop}\label{propRelacionOrdenes}
Let $\M$ and $\L$ be sequences.
\begin{itemize}
\item[(i)] If  $\bm\simeq\bl$ then  $\M\approx\L$.
\item[(ii)] If $\M$ and $\L$ are (lc) and one of them is (mg), then $\M\approx\L$ amounts to $\bm\simeq\bl$.
In particular, for strongly regular sequences one may interchange $\simeq$ and $\approx$.
\end{itemize}
\end{prop}

\begin{rema}\label{remaCambioSucesNoLogarConvex}
Clearly, property (mg) is preserved by the relation $\approx$, and property (snq) is preserved by $\simeq$. So, if a sequence $\M$ is (mg) and (snq), and another sequence $\L$ is (lc) and such that $\bl\simeq\bm$, then $\L$ is strongly regular. In particular, whenever $\M$ is (mg) and (snq), and $\bm$ is eventually nondecreasing, it is easy to construct a strongly regular sequence $\L$ such that $\bl\simeq\bm$ and, in fact,
$\ell_p=m_p$ for every $p$ greater than or equal to some suitable $p_0$.
\end{rema}

\begin{exam}\label{examstroregusequ}
\begin{itemize}
\item[(i)] The best known example of strongly regular sequence is the \textit{Gevrey sequence of order $\a>0$}, $\M_{\a}=(p!^{\a})_{p\in\N_{0}}$.
\item[(ii)] The sequences $\M_{\a,\b}=\big(p!^{\a}\prod_{m=0}^p\log^{\b}(e+m)\big)_{p\in\N_0}$, where $\a>0$ and $\b\in\R$, are strongly regular (in case $\b<0$, the sequence has to be suitably modified according to Remark~\ref{remaCambioSucesNoLogarConvex}).
\item[(iii)] For $q>1$, $\M=(q^{p^2})_{p\in\N_0}$ is (lc) and (snq), but not (mg).
\end{itemize}
\end{exam}

The following result, given by H.-J Petzsche~\cite{Pet}, will be important later on.

\begin{pro}[\cite{Pet}, Prop.\ 1.1]\label{petzche11}
 Let $\M$ be a sequence such that the sequence $\M^*$ given by $M^*_p:=p!M_p$, $p\in\N_0$, is logarithmically convex. Then,
 the following statements are equivalent:
 \begin{enumerate}[(i)]
 \item $\M$ verifies (snq).
 \item There exists $k\in\N$, $k\ge 2$, such that $\M$ verifies
 \begin{equation}\label{condition.liminf.mkp.over.mp.great.1}
  \liminf_{p\ri\oo} \frac{m_{kp}}{m_p}>1.
 \end{equation}
 \item There exists a logarithmically convex sequence $\mathbb{H}$ such that $\h\simeq \m$ and
 \begin{equation}\label{condition.m2p.over.mp.great.1}
  \inf_{p\geq1} \frac{h_{2p}}{h_p}>1.
 \end{equation}

 \end{enumerate}
\end{pro}

We deduce the following useful corollary, whose proof is included for the convenience of the reader.

\begin{coro}[\cite{Pet}, Corol.\ 1.3(a)] \label{petzche13}
 Let $\M$ be a (lc) sequence verifying (snq), then there exist $\ep>0$ and a sequence $\mathbb{L}$ such that $\l\simeq\m$ and
 the sequence $(L_p (p!)^{-\ep})_{p\in\N_0}$ is (lc) and  verifies (snq).
\end{coro}

\begin{proof}
The sequence $ (p! M_p )_{p\in\N_0}$ satisfies the hypotheses in Proposition~\ref{petzche11}, so by (iii) there exists $\mathbb{H}$ such
that $\H$ is (lc), satisfies (\ref{condition.m2p.over.mp.great.1}) and $ \bf{h} \simeq \m $. So, there exists $\xi>0$ such that for every $p\in\N$ we have
$$\frac{h_{2p}}{ h_p} \geq 2^{\xi}.$$
We fix $0<\ep<\min(1,\xi)$, let us show that $(h_p (p+1)^{-\ep})_{p\in\N_0}$ also satisfies (\ref{condition.m2p.over.mp.great.1}).
The sequence $(p+1)/(2p+1)$ decreases towards $1/2$ and we have
$$\frac{h_{2p}(2p+1)^{-\ep}}{h_p (p+1)^{-\ep}} \geq \frac{h_{2p}(2)^{-\ep}}{h_p} \geq \frac{2^\xi}{2^\ep}>1.  $$
As $\H$ is (lc), the
sequence $(h_p (p+1)^{1-\ep})_{p\in\N}$ also is, because $1-\ep>0$. As
condition (\ref{condition.m2p.over.mp.great.1}) implies (\ref{condition.liminf.mkp.over.mp.great.1}), we may apply Proposition~\ref{petzche11}(iii) to the sequence $(H_p (p!)^{-\ep})_{p\in\N_0}$, so there exists $\A$ that is (lc), satisfies (\ref{condition.m2p.over.mp.great.1})
and ${\bf a} \simeq (h_p (p+1)^{-\ep})_{p\in\N_0} $. We define $L_p:= A_p (p!)^{\ep}$ and we observe that
$$\m \simeq \h \simeq (a_p (p+1)^{\ep})_{p \in \N_0} = \bl. $$
The sequence $(L_p (p!)^{-\ep})_{p\in\N_0} =( A_p)_{p\in\N_0} $ is logarithmically convex. Finally, by Remark~\ref{remaCambioSucesNoLogarConvex} we see
that $\mathbb{L}$ verifies (snq).
\end{proof}

One may associate with a strongly regular sequence $\bM$ the function
\begin{equation}\label{equadefiMdet}
M(t):=\sup_{p\in\N_{0}}\log\big(\frac{t^p}{M_{p}}\big),\quad t>0;\qquad M(0)=0,
\end{equation}
which is a non-decreasing continuous map in $[0,\infty)$ with $\lim_{t\to\infty}M(t)=\infty$.
Indeed,
\begin{equation*}
M(t)=\left \{ \begin{matrix}  p\log t -\log(M_{p}) & \mbox{if }t\in [m_{p-1},m_{p}),\ p=1,2,\ldots,\\
0 & \mbox{if } t\in [0,m_{0}), \end{matrix}\right.
\end{equation*}
and one can easily check that $M$ is convex in $\log t$, i.e., the map $t\mapsto M(e^t)$ is convex in $\mathbb{R}$.
We also observe that
\begin{equation}
 M(m_p)=\log \left( \frac{m^p_p}{M_p}\right), \qquad p\in \N_0. \label{equation.M.in.mp}
\end{equation}

\subsection{Asymptotic expansions and ultraholomorphic classes}\label{subsectCarlemanclasses}

We recall the concept of asymptotic expansion in the case that the remainders are controlled in terms of a sequence of positive real numbers $\M=(M_p)_{p\in\N_0}$.

\begin{defi}
We say a complex holomorphic function $f$ in a sectorial region $G$ admits the formal power series $\hat{f}=\sum_{k=0}^{\infty}a_{k}z^{k}\in\C[[z]]$ as
its $\bM-$\emph{asymptotic expansion} in $G$ (when the variable tends to 0) if for every $T\ll G$ there
exist $C_T,A_T>0$ such that for every $p\in\N_0$ and $z\in T$, one has
\begin{equation*}\Big|f(z)-\sum_{k=0}^{p-1}a_kz^k \Big|\le C_TA_T^pM_{p}|z|^p.
\end{equation*}
We will write $f\sim_{\bM}\sum_{k=0}^{\infty}a_kz^k$ in $G$. $\tilde{\mathcal{A}}_{\M}(G)$ stands for the space of
functions admitting $\bM-$asymptotic expansion in $G$, and it is called a \emph{Carleman ultraholomorphic class} (in the sense of Roumieu).
\end{defi}

As a consequence of Taylor's formula and Cauchy's integral formula for the derivatives, we have the following result (see \cite{Balserutx} for a proof in the Gevrey case, which may be easily adapted to this more general situation).

\begin{prop}\label{propcotaderidesaasin}
Let $G$ be a sectorial region and $f\colon G\to\C$ a holomorphic function. The following are equivalent:
\begin{itemize}
\item[(i)] $f\in\tilde{\mathcal{A}}_{\M}(G)$.
\item[(ii)] For every $T\ll G$ there exist $C_T,A_T>0$ such that for every $p\in\N_0$ and $z\in T$, one has
    $$
    |f^{(p)}(z)|\le C_TA_T^p p!M_p.
    $$
    \end{itemize}
In case any of the previous holds, for every $p\in\N_{0}$ one may define
$$f^{(p)}(0):=\lim_{z\in T,z\to0 }f^{(p)}(z)\in\C,$$
which is independent of $T\ll G$, and one has $f\sim_{\M}\sum_{p\in\N_0}\frac{1}{p!}f^{(p)}(0)z^p$ in $G$.
\end{prop}

In accordance with the ultraholomorphic classes we define the classes of sequences
$$\Lambda_{\M}=\Big\{\bmu=(\mu_{p})_{p\in\N_{0}}\in\C^{\N_{0}}: \textrm{there exist $C,A>0$ with }|\mu_p|\le CA^{p}p!M_{p},\ p\in\N_{0}\Big\}.$$
Then, it is clear that the map $\tilde{\mathcal{B}}:\tilde{\mathcal{A}}_{\M}(G)\longrightarrow \Lambda_{\M}$ given by
\begin{equation*}
\tilde{\mathcal{B}}(f):=(f^{(p)}(0))_{p\in\N_{0}},
\end{equation*}
is well defined, and it will be called the \textit{asymptotic Borel map}.

\begin{defi}
A function $f\in\tilde{\mathcal{A}}_{\M}(G)$ is said to be \textit{flat} if $\tilde{\mathcal{B}}(f)$ is the null sequence or, in other words, $f\sim_{\M}\hat{0}$, where $\hat{0}$ denotes the null power series.
\end{defi}

\subsection{Quasianalyticity}

We are interested in characterizing those ultraholomorphic classes in which the asymptotic Borel map is injective.

\begin{defi}
Let $G$ be a sectorial region and $\M=(M_{p})_{p\in\N_{0}}$ be a sequence of positive numbers. We say that $\mathcal{\tilde{A}}_{\M}(G)$
is  \textit{quasianalytic} if it does not contain nontrivial flat functions.
\end{defi}

\begin{rema} By a simple rotation, we see that the bisecting direction $d$ of the sectorial region $G$ is irrelevant in the study of
quasianalyticity. This allow us to consider only sectorials regions bisected by the direction $d=0$, which will be denoted by $G_\ga$ whenever their opening equals $\pi\ga$, no matter what their specific shape is.
\end{rema}

In order to obtain our next result, Theorem~\ref{theo.partial.version.general.Watson.Lemma.following.Mandelbrojt}, we will use the following generalization of Watson's Lemma, given
by S. Mandelbrojt in~\cite{mandelbrojt}. A different approach may be found in~\cite{SanzFlat}, where we used a result of B. I. Korenbljum~\cite{korenbljum} on quasianalyticity for ultraholomorphic classes on sectors with uniformly bounded derivatives, a setting which we are not interested in considering here.

\begin{theo}[\cite{mandelbrojt},\ Section\ 2.4.III]
 Let $\M$ be a logarithmically convex sequence with $\lim_{p\ri\oo} m_p=\oo$, $H=\{z\in\C: \Re(z)>0 \}$ and $\alpha>0$. The following statements are equivalent:
 \begin{enumerate}[(i)]
  \item $\displaystyle \sum_{p=0}^{\oo} \left(\frac{1}{m_p} \right)^{1/\a}$ diverges,
  \item If $f\in \mathcal{H}(H)$ and there exist $A,C>0$ such that
  $$|f(z)|\leq \frac{CA^pM_p}{|z|^{\a p}}, \qquad z\in H, \quad p\in \N_0,$$
  then $f$ identically vanishes.
 \end{enumerate}

\end{theo}

We need also to recall now
the definition of exponent of convergence of a sequence and how it may be computed.

\begin{prop}[\cite{HOLL},\ p.\ 65]
Let $(c_p)_{p\in\N_0}$ be a nondecreasing sequence of positive real numbers tending to infinity. The \textit{exponent of convergence}
of $(c_p)_p$ is defined as
$$
\lambda_{(c_p)}:=\inf\{\mu>0:\sum_{p=0}^\infty \frac{1}{c_p^{\mu}}\textrm{ converges}\}
$$
(if the previous set is empty, we put $\lambda_{(c_p)}=\infty$). Then, one has
\begin{equation}\label{equaexpoconv}
\lambda_{(c_p)}=\limsup_{p\to\infty}\frac{\log(p)}{\log(c_p)}.
\end{equation}
\end{prop}

We are ready to characterize quasianalyticity in terms of the exponent of convergence $\lambda_{(m_p)}$.

\begin{theo}(generalized Watson's lemma, partial version) \label{theo.partial.version.general.Watson.Lemma.following.Mandelbrojt}
Let $\ga$ be a positive constant and $G_\ga$ a sectorial region, the following statements hold:
\begin{enumerate}[(i)]
 \item If $\ga>1/\lambda_{(m_p)}$,  then  $\mathcal{\tilde{A}}_{\M}(G_\ga)$ is quasianalytic.
 \item If $\ga<1/\lambda_{(m_p)}$,  then  $\mathcal{\tilde{A}}_{\M}(G_\ga)$ is not quasianalytic.
\end{enumerate}
\end{theo}

\begin{proof}
(i) Assume that $\ga>1/\lambda_{(m_p)}$.  We take $f\in  \mathcal{\tilde{A}}_{\M}(G_\ga)$ with $f\sim_{\M} \hat{0}$, and consider a proper bounded sector $T\ll G$, where $T=S(0,\b,r)$ with $\ga>\b>1/\lambda_{(m_p)}$. By the definition of
$\M-$asymptotic expansion, there exist $C_T,A_T>0$ such that
$$|f(z)|\leq C_TA_T^p M_p |z|^p, \qquad z\in T, \quad p\in\N_0.$$
We consider the transformation $z(w)=1/(w+(1/r)^{1/\b})^\b$, which maps $H$ into a region $D\subseteq T$, and the holomorphic function
$g:H\ri\C$ defined by $g(w):=f(z(w))$. As for every $w\in H$ we have $|w+(1/r)^{1/\b}|>|w|$, we deduce that
$$|g(w)|=|f(z(w))|\leq \frac{C_TA_T^p M_p}{|(w+(1/r)^{1/\b})^\b|^p}\leq \frac{C_TA_T^p M_p}{|w|^{\b p}}, \qquad w\in H, \quad p\in \N_0.$$
So $g$ verifies the bounds in Theorem~\ref{theo.partial.version.general.Watson.Lemma.following.Mandelbrojt}.(ii). Since $\b>1/\lambda_{(m_p)}$,
by the definition of the exponent of convergence we know that
$$ \sum_{p=0}^{\oo} \left(\frac{1}{m_p} \right)^{1/\b}=\oo.$$
Consequently, by Theorem~\ref{theo.partial.version.general.Watson.Lemma.following.Mandelbrojt} we obtain that $g\equiv0$ and $f\equiv0$, so that  $\mathcal{\tilde{A}}_{\M}(G_\ga)$ is quasianalytic.

\noindent (ii) Assume that $\ga<1/\lambda_{(m_p)}$, what implies, by the very definition of the exponent of convergence, that
$$ \sum_{p=0}^{\oo} \left(\frac{1}{m_p} \right)^{1/\ga}<\oo.$$
Then, by Theorem~\ref{theo.partial.version.general.Watson.Lemma.following.Mandelbrojt} there exist
$f\in \mathcal{H}(H)$, not identically zero and such that there exist $A,C>0$ with
$$|f(z)|\leq \frac{CA^p M_p}{|z|^{\ga p}}, \qquad z\in H, \quad p\in \N_0.$$
The function $w\to w^{-1/\ga}$ maps the sector $S_\ga$ into $H$. We consider the function $g(w)=f(w^{-1/\ga})$, holomorphic in $S_\ga$, and
we have that
$$|g(w)|=|f(w^{-1/\ga})|\leq C A^p M_p |w|^p, \qquad w\in S_\ga, \quad p\in \N_0.$$
Consequently, $g\not\equiv 0$ and  $g\sim_{\M} \hat{0}$ in $S_\ga$. We observe that the restriction of $g$ to $G_\ga$ is a nontrivial flat function,
and we deduce that $\mathcal{\tilde{A}}_{\M}(G_\ga)$ is not quasianalytic.
\end{proof}

The following definition is natural.

\begin{defi}
For a strongly regular sequence $\M$, we define the \textit{order of quasianalyticity}
$$\omega(\M):=\inf \{\ga>0:\mathcal{\tilde{A}}_{\M}(G_{\gamma}) \textrm{ is quasianalytic}\}.
$$
\end{defi}

From Theorem~\ref{theo.partial.version.general.Watson.Lemma.following.Mandelbrojt} we immediately deduce the following result (except for the last statement, which was proved in~\cite{SanzFlat}).

\begin{coro}\label{coroOrderM}
For a strongly regular sequence $\M$, we have
\begin{equation*}
\omega(\M)=\liminf_{p\to\infty}
\frac{\log(m_{p})}{\log(p)}=\frac{1}{\lambda_{(m_p)}}.
\end{equation*}
Moreover, $\omega(\M) \in (0,\oo)$.
\end{coro}

\begin{rema}
Consider a pair of equivalent sequences $\bM$ and $\bM'$. For any sectorial region~$G$,
 $\tilde{\mathcal{A}}_{\M}(G)$ and $\tilde{\mathcal{A}}_{\M'}(G)$ coincide, so it is clear that $\omega(\M)=\omega(\M')$.
\end{rema}

In order to study quasianalyticity in case the opening of the sectorial region equals the limiting value $\pi \o(\M)$, we need to introduce the notion of proximate order,
appearing in the theory of
growth of entire functions and developed, among others, by E. Lindel\"of, G. Valiron~\cite{Valiron42}, B. Ja. Levin~\cite{Levin}, A. A. Goldberg and I. V. Ostrosvkii~\cite{GoldbergOstrowskii} and L. S. Maergoiz~\cite{Maergoiz}.

\begin{den}\label{OAD:1}
We say a real function $\ro(r)$, defined on $(c,\oo)$ for some $c\ge 0$, is a {\it proximate order},
if the following hold::
 \begin{enumerate}[(A)]
  \item $\ro$ is continuous and piecewise continuously differentiable in $(c,\oo)$ (meaning that it is differentiable except possibly at a sequence of points, tending to infinity, at any of which it is continuous and has distinct finite lateral derivatives),\label{OA1:1}
  \item $\ro(r) \geq 0$ for every $r>c$,\label{OA2:1}
  \item $\lim_{r \ri \oo} \ro(r)=\ro< \oo$, \label{OA3:1}
  \item $\lim_{r  \ri \oo} r \ro'(r) \ln r = 0$. \label{OA4:1}
 \end{enumerate}
\end{den}

Under the assumption that the auxiliary function $d_{\M}(t)=\log(M(t))/\log t$, where $M(t)$ was defined in \eqref{equadefiMdet}, is a proximate order, the second author~\cite{SanzFlat} provided flat functions in the optimal sectors $S_{\o(\M)}$ by applying results of L. S.~Maergoiz~\cite{Maergoiz}. So, we have the following final statement in this respect.

\begin{theo}[Watson's Lemma, \cite{SanzFlat}]\label{coroWatsonlemma}
Suppose $\M$ is strongly regular and such that $d_{\M}(t)$ is a proximate order, and let $\ga>0$ and a sectorial region $G_\ga$ be given. The following statements are equivalent:
\begin{itemize}
\item[(i)] $\tilde{\mathcal{A}}_{\M}(G_{\gamma})$ is quasianalytic.
\item[(ii)] $\gamma>\omega(\M)$.
\end{itemize}
\end{theo}

\section{Strongly Regular Sequences and proximate orders}

In this section we will present different characterizations of $d_{\M}$ being a proximate order for a strongly regular sequence $\M$.

\noindent We start by recalling the following definitions and facts, mainly taken from
 the book of A. A. Goldberg and I. V. Ostrovskii~\cite{GoldbergOstrowskii}.

\begin{defi}[\cite{GoldbergOstrowskii}, p.\ 43]
Let $\a(t)$ be a nonnegative and nondecreasing function in $(c,\infty)$ for some $c\ge 0$ (we write $\a\in\Lambda$).
The \textit{order} of $\a$ is defined as
$$
\rho=\rho[\a]:=\limsup_{t\to\infty}\frac{\log^+\a(t)}{\log t}\in[0,\infty],
$$
where $\log^+=\max(\log,0)$. $\a(t)$ is said to have finite order if $\rho<\infty$.
\end{defi}

We are firstly interested in determining the order of the function $M(t)\in\Lambda$ (defined in (\ref{equadefiMdet}) for every strongly regular sequence $\M$).
We will need the following fact, which can be found in~\cite{mandelbrojt}: if we consider the \textit{counting function} for the sequence of quotients $\bm$,
$\nu:(0,\infty)\to\N_0$ given by
\begin{equation}\label{equadefinuder}
\nu(t):=\#\{j:m_j\le t\},
\end{equation}
then one has that
\begin{equation}\label{equarelaMdetnuder}
M(t)=\int_0^t\frac{\nu(r)}{r}\,dr,\qquad t>0.
\end{equation}

The following result consists of the right assertions included in Theorem 2.14 in~\cite{SanzFlat}, where, as aforementioned, it was incorrectly obtained that $\lim_{t\to\infty}\log M(t)/\log t$ exists.

\begin{theo} 
Let $\M$ be strongly regular. 
Then, the order of $M(t)$ is given by
\begin{equation}\label{equaordeMdet}
\rho[M]=\limsup_{t\to\infty}\frac{\log M(t)}{\log t}=\limsup_{p\to\infty}\frac{\log(p)}{\log(m_{p})}.
\end{equation}
Consequently, by Corollary~\ref{coroOrderM} we have
\begin{equation*}
\rho[M]=\frac{1}{\omega(\M)}\in(0,\infty).
\end{equation*}
\end{theo}

\begin{proof}
The first expression for $\rho[M]$ comes from the very definition of the order because $M(t)>1$ for $t$ large enough.
 For the second one, 
we take into account the link given in (\ref{equarelaMdetnuder}) between $M(t)$ and  the counting function $\nu(t)$, 
which also belongs to $\Lambda$. We may apply Theorem 2.1.1 in~\cite{GoldbergOstrowskii} to deduce that the order of $M(t)$ equals that of $\nu(t)$.
Now, from Theorem 2.1.8 in~\cite{GoldbergOstrowskii} we know that the order of $\nu(t)$ is in turn the exponent of convergence of $\bm$, given by the formula in (\ref{equaexpoconv}).
\end{proof}

The function $d_{\M}(t)$ always verifies conditions (\ref{OA1:1}) and (\ref{OA2:1}) of proximate orders (see Definition~\ref{OAD:1}).
In order to study condition (\ref{OA3:1}), we will use Proposition~\ref{petzche13} to prove the following auxiliary result.

\begin{lemma}\label{lemmaliminfcierto}
Let $\M$ be a strongly regular sequence, then there exists a strongly regular sequence $\L$ such that $\L\approx\M$ and
\begin{equation}\label{des.liminf.L}
\liminf_{p\ri\oo} \log\left(\frac{\ell_p}{L^{1/p}_p}\right)>0,
\end{equation}
where $\bl=(\ell_p)_{p\in\N_0}$ is the sequence of quotients associated with $\L$. Moreover, if $L(t)$ is the associated function of $\L$ we deduce that
\begin{equation}
\lim_{p\ri\oo} \frac{\log(L(\ell_p))}{\log(p)}=1.\label{limit.logfunctionL.over.p.equal.1}
\end{equation}

\end{lemma}
\begin{proof}
Since $\M$ is (lc) and verifies (snq), by Proposition~\ref{petzche13} there exist $\ep>0$ and a sequence of positive numbers $(L_p)_{p\geq0}$ such that
$\m\simeq\bl$ and $(L_p (p!)^{-\ep})_{p\in\N_0}$ is (lc)
and verifies (snq). Remark~\ref{remaCambioSucesNoLogarConvex}  implies that
$\L$ also satisfies (mg) and (snq). Moreover, as $(L_p (p!)^{-\ep})_{p\in\N_0}$  is (lc), $\L$  also is: 
$$\ell_p=\ell_p (p+1)^{\ep-\ep} \leq \frac{\ell_{p+1}}{(p+2)^\ep} (p+2)^\ep = \ell_{p+1},\quad p\in\N_0.
$$
Now let us show \eqref{des.liminf.L}.
Since the sequence $(\ell_p (p+1)^{-\ep})_{p\in\N_0}$ is nondecreasing, for every  $p,k\in\N_0$ with $p>k$ we have
 $$\ell_k\leq \frac{(k+1)^\ep}{(k+2)^\ep} \ell_{k+1} \leq \ell_p \prod_{j=k}^{p-1} \left( \frac{j+1}{j+2} \right)^\ep= \ell_p \prod_{j=k+1}^{p} \left( \frac{j}{j+1} \right)^\ep.$$
 Consequently,
 $$ L_p=\ell_0\ell_1\cdots \ell_{p-1}\leq \ell^p_p \prod_{j=1}^{p} \left( \frac{j}{j+1} \right)^{\ep j }.$$
We observe that
$$\prod_{j=1}^{p} \left( \frac{j^j}{(j+1)^{j+1}} (j+1) \right)= \frac{(p+1)!}{(p+1)^{p+1}}, $$
so we can write
$$ L_p\leq \ell^p_p  \left(\frac{(p+1)!}{(p+1)^{p+1}} \right)^{\ep }.$$
 By Stirling's formula, $\lim_{p\ri\oo} (p! 2^p) /  p^p = 0<1$, and there exists $p_0\in\N$ such that for every $p\geq p_0$ one has
 $$ L_p\leq \ell^p_p  \left( \frac{1}{2^{p+1}}\right)^{\ep }.$$
 It is clear then that
 $$  \log\left( \frac{\ell_p}{L^{1/p}_p}\right)\geq \ep  \frac{p+1}{p} \log(2) \geq \ep \log(2) >0, \qquad p\geq p_0,$$
and we conclude that $\L$ verifies~(\ref{des.liminf.L}). Finally, using~\eqref{eqmg_mppMp} and~(\ref{equation.M.in.mp}),
there  exist $a>1$ such that
 $$\frac{L(\ell_p)}{p}=\log \left(\frac{\ell_p}{L^{1/p}_p}\right) \leq \log (a)$$
 for every $p\in\N$. For $p\geq p_1$ large enough we know $\log L(\ell_p)>0$, and taking logarithms in the above expression we see that
 \begin{equation}\label{inequality.logM.sup}
\log L(\ell_p) \leq \log (p) + \log(\log(a)).
 \end{equation}
On the other hand, by~(\ref{des.liminf.L}) there exist $\ep_0>0$ and $p_2\in \N$ such that
 $$\log \left(\frac{\ell_p}{L^{1/p}_p}\right) \geq \ep_0 >0$$
 for $p\geq p_2$, so we deduce that
 \begin{equation}\label{inequality.logM.inf}
  \log L(\ell_p) \geq \log (p) + \log(\ep_0).
 \end{equation}
Using~(\ref{inequality.logM.sup}) and~(\ref{inequality.logM.inf}) we see that \eqref{limit.logfunctionL.over.p.equal.1} holds.

 \end{proof}

Thanks to the previous lemma, we can give the following characterization of condition~(\ref{OA3:1}).

\begin{theo}\label{teorcondicion3caracterizacion}
Let $\M$ be a strongly regular sequence. Then
$d_{\M}(t)$ verifies condition~(\ref{OA3:1}) if, and only if,
\begin{equation}\label{limit.logmp.over.logp.omega}
    \lim_{p\ri \oo} \frac{\log (m_p)}{\log(p)}=\o(\M).
\end{equation}
\end{theo}

\begin{proof}
By Lemma~\ref{lemmaliminfcierto} there exists a strongly regular
sequence $\L$ such that $\L\approx\M$ and $\L$ verifies~(\ref{limit.logfunctionL.over.p.equal.1}).
We consider  $\bl=(\ell_p)_{p\in\N_0}$, the sequence of quotients of $\L$, and $L(t)$, its
  associated function.

According to \eqref{equaordeMdet}, in case $d_{\M}(t)$ verifies condition~(\ref{OA3:1}) we will have that $\lim_{t\ri\oo} d_{\M}(t)=1/\o(\M)$. There exists $H>0$ such that for every $p\in\N_0$,
$H^{-p}M_p\leq L_p \leq H^p M_p$, and so
\begin{equation}\label{equaDesigMtLtEquiv}
M(\frac{t}{H})=\sup_{p\in\N_0} \log\left(\frac{t^p}{H^pM_p}\right)\leq L(t)=\sup_{p\in\N_0} \log\left(\frac{t^p}{L_p}\right) \leq \sup_{p\in\N_0}
\log\left(\frac{H^p t^p}{M_p}\right)=M(Ht).
\end{equation}
Consequently,
\begin{equation}\label{equaLimitddeL}
\lim_{t\ri\oo} d_{\L}(t)=\lim_{t\ri\oo}\frac{ \log(L(t))}{\log(t)}= 1/\o(\M).
\end{equation}
Using~(\ref{limit.logfunctionL.over.p.equal.1}) and \eqref{equaLimitddeL} we see that
$$\lim_{p\ri\oo} \frac{\log (\ell_p)}{\log(p)}=\lim_{p\ri\oo} \frac{\log (\ell_p)}{\log L(\ell_p)} \frac{\log L(\ell_p)}{\log(p)} = \frac{1}{\ro[L]}=\o(\L)=\o(\M).$$
As $\M$ and $\L$ are strongly regular, we have $\m\simeq\l$, i.e. there exists $a>1$ such that $a^{-1} m_p\leq l_p \leq a m_p$ for every $p\in\N_0$. It is clear then that
$$\lim_{p\ri\oo} \log(m_p)/\log(p)=\lim_{p\ri\oo} \log(\ell_p)/\log(p)=\o(\M).
$$

Conversely, let us assume now that $\lim_{p\ri\oo} \log(m_p)/\log(p)=\o(\M)$. As before, this amounts to
$$\lim_{p\ri \oo} \frac{\log (\ell_p)}{\log(p)}=\o(\M).$$
Using~(\ref{inequality.logM.sup}) and~(\ref{inequality.logM.inf}), for every $t\in(\ell_{p-1},\ell_p)$ and $p$ large enough, we see that
  $$\frac{\log (p-1) + \log(\ep)}{\log(\ell_{p})} \leq \frac{\log(L(\ell_{p-1}))}{\log(\ell_{p})}\leq \frac{\log(L(t))}{\log(t)}\leq \frac{\log(L(\ell_p))}{\log(\ell_{p-1})}\leq \frac{\log (p) + \log(\log(a))}{\log(\ell_{p-1})}. $$
   Taking limits we deduce that
  $\lim_{t\ri\oo} d_{\L}(t)= 1/\o(\M)$, and thanks to \eqref{equaDesigMtLtEquiv} we conclude that $\lim_{t\ri\oo} d_{\M}(t)=1/\o(\M)$, as desired.
 \end{proof}

In the next result we assume that condition (\ref{OA3:1}) for $d_{\M}$ holds, and obtain a characterization of condition (\ref{OA4:1}). This result follows the ideas in the proof of Proposition 4.9 in \cite{SanzFlat}, but takes into account that $d_{\M}$ need not to be monotone, as it was thought to be true at that moment.

\begin{prop}\label{propcaracdderordenaprox}
Let $\M$ be a strongly regular sequence. If we assume that $\lim_{t\ri\oo} d_{\M}(t)=1/\o(\M)$,
the following are equivalent:
\begin{itemize}
\item[(i)] $d_{\M}(t)$ is a proximate order, i.e., it verifies (\ref{OA4:1}),
\item[(ii)] $\displaystyle\lim_{p\to\infty}\frac{p}{M(m_{p})}=\frac{1}{\omega(\M)}=\ro[M]$.
\end{itemize}
\end{prop}

\begin{proof}
We have that
$$d_{\M}'(t)=\frac{M'(t)}{\log(t)M(t)}-\frac{d_{\M}(t)}{t (\log(t))}$$
whenever it exists. If $t\in(m_{p-1},m_p)$,
$$d_{\M}'(t)=\frac{p}{t\log(t) M(t)}-\frac{d_{\M}(t)}{t \log(t)}=
\frac{1}{t\log(t)}\left(\frac{\nu(t)}{M(t)}-d_{\M}(t)\right),$$
where $\nu(t)$ is the function defined in~(\ref{equadefinuder}). For convenience,
write $b(t)=td_{\M}'(t)\log(t)$. Then $\lim_{t\ri\oo} b(t)=0$ (i.e., (\ref{OA4:1})
is satisfied) if, and only if,
$$\lim_{t\ri \oo} \Big(\frac{\nu(t)}{M(t)} - d_{\M}(t)\Big)=0.$$
Since we are assuming that $\lim_{t\ri\oo} d_{\M}(t) = 1/\o(\M)$, the previous equality holds if, and only if,
\begin{equation}\label{equaLimitNuOverM}
\lim_{t\ri \oo} \frac{\nu(t)}{M(t)}= \frac{1}{\o(\M)}.
\end{equation}
Observe that $M(t)$ is nondecreasing, and $\nu(t)$ is constant in $(m_{p-1},m_p)$, so $\nu(t)/M(t)$ is nonincreasing in $(m_{p-1},m_p)$ and we deduce that
$$\frac{p}{M(m_p)}\leq\frac{\nu(t)}{M(t)}\leq \frac{p}{M(m_{p-1})} =\frac{p-1}{M(m_{p-1})} + \frac{1}{M(m_{p-1})}, \qquad t\in(m_{p-1},m_p).$$
As $\lim_{p\ri\oo} 1/M(m_p) =0$, from the previous inequalities it is clear that \eqref{equaLimitNuOverM} holds if, and only if,
$$\lim_{p\ri \oo} \frac{p}{M(m_p)}= \frac{1}{\o(\M)},$$
as desired.
\end{proof}

\begin{rema}
Taking into account~(\ref{equation.M.in.mp}), condition (ii) in Proposition \ref{propcaracdderordenaprox} may be written as
 \begin{equation}\label{limit.beta.infinity.omega}
  \lim_{p\ri\oo} \log\left(\frac{m_p}{M_p^{1/p}}\right)=\o(\M).
 \end{equation}
\end{rema}

In order to conclude this section we will show that~(\ref{limit.beta.infinity.omega}) is equivalent to~(\ref{limit.logmp.over.logp.omega}).

\begin{nota}\label{notacionAlfaBeta}
For convenience, 
given a logarithmically convex
sequence $\M$ we consider
$$\a_p:=\log(m_p),\quad p\in\N_0; \qquad \b_0:=\a_0, \qquad \b_p:=\log\left(\frac{m_p}{M^{1/p}_p}\right)=\frac{M(m_p)}{p}, \quad p\geq1.$$
With this notation, condition~(\ref{limit.logmp.over.logp.omega}) is
\begin{equation}\label{equaCondEquivLogmpOverLogpOmega}
\lim_{p\ri\oo} \frac{\a_p}{\log(p)}=\o(\M),
\end{equation}
and (\ref{limit.beta.infinity.omega}) amounts to
\begin{equation}\label{equaCondEquivBetaOmega}
\lim_{p\ri\oo} \b_p=\o(\M),
\end{equation}
so we will prove that \eqref{equaCondEquivLogmpOverLogpOmega} and \eqref{equaCondEquivBetaOmega} are equivalent.
\end{nota}

\begin{lemma}
 Given a logarithmically convex
 sequence $\M$ we have
 \begin{align}
  \b_p=&\a_p-\frac{1}{p}\sum^{p-1}_{k=0}\a_k,\qquad p\in\N_0, \label{eqbetaalfa}\\
  \a_p=&\sum^{p-1}_{k=0}\frac{\b_k}{k+1}+\b_p, \qquad p\in\N_0.\label{eqalfabeta}
 \end{align}
 \end{lemma}
 \begin{proof}
  From the definition of $(\b_p)_{p\in\N}$ we have that $\b_0=\a_0$, and for $p\in\N$
  $$\b_p= \log\left(\frac{m_p}{M^{1/p}_p}\right)=\log(m_p)-\frac{1}{p}\log(m_0m_1\cdots m_{p-1})=\a_p-\frac{1}{p}\sum^{p-1}_{k=0}\a_k,$$
    then~(\ref{eqbetaalfa}) holds.

    For the proof of \eqref{eqalfabeta} we apply induction. It clearly holds for $p=0$, and if we admit its validity for some $p\in\N_0$, then
    \begin{multline*}
    \a_{p+1}=\log(m_{p+1})=\a_p+\log\big(\frac{m_{p+1}}{m_p}\big)= \sum^{p-1}_{k=0}\frac{\b_k}{k+1}+\b_p+\log\big(\frac{m_{p+1}}{m_p}\big)\\
    =\sum^{p}_{k=0}\frac{\b_k}{k+1}+\frac{p}{p+1}\b_p+\log\big(\frac{m_{p+1}}{m_p}\big).
    \end{multline*}
So, we are done if it holds that
$$
\frac{p}{p+1}\b_p+\log\big(\frac{m_{p+1}}{m_p}\big)=\b_{p+1},$$
but this equality can be easily checked by direct manipulation.
 \end{proof}

\begin{rema}\label{remaRieszMethod}
By using~(\ref{eqalfabeta}) and Stolz's criterion we easily deduce that~\eqref{equaCondEquivBetaOmega} implies \eqref{equaCondEquivLogmpOverLogpOmega}.
To prove the converse, we need to recall the Riesz weighted average summability method for sequences, also called sequence summation by logarithmic means.
\end{rema}

\begin{defi}
A numerical sequence $(s_k)_{k\in\N}$ of complex numbers is said to be \textit{logarithmic summable or Riesz summable of order $1$}, briefly $(L, 1)-$summable , if there exists some $A\in\C$ such that
\begin{equation}\label{equationdefrieszsummability}
 \lim_{p\ri\oo} \frac{1}{H_p}\sum^p_{k=1}\frac{s_k}{k}=A, \text{where} \qquad H_p:=\sum^p_{k=1} \frac{1}{k} \sim \log(p)
\end{equation}
(for two sequences $(a_p)$ and $(b_p)$ of positive numbers we write $a_p\sim b_p$ if
$\lim_{p\ri\oo} a_p/b_p =1$).
\end{defi}

\begin{rema}
The method is regular, that is, if the ordinary limit
exists, then the limit in~(\ref{equationdefrieszsummability}) also exists and with the same value. As we will see, we are interested in applying the converse implication, which is not true in general. There are several standard results in the literature, called Tauberian theorems for Riesz summability (see for example the book of J. Boos~\cite{Boos}), allowing to reverse the implication under some suitable conditions, but none of these results seemed to meet our needs. Instead, we will use the following recent and powerful result of F. Moricz, which characterizes those $(L, 1)-$summable sequences that are convergent.
\end{rema}

\begin{theo}[\cite{Moricz},\ Th.\ 5.1]\label{theoMoricz}
If a sequence $(s_k)$ of real numbers is $(L, 1)-$summable to some $A\in\R$, then the ordinary limit exists (with the same value) if, and only if,
\begin{equation}\label{eqmoricz1}
 \limsup_{\lambda\ri1^+}\liminf_{p\ri\oo} \frac{1}{(\lfloor p^\lambda\rfloor-p)H_p}\sum_{k=p+1}^{\lfloor p^\lambda\rfloor}\frac{s_k-s_p}{k}\geq 0
\end{equation}
and
\begin{equation}\label{eqmoricz2}
 \limsup_{\lambda\ri1^-}\liminf_{p\ri\oo} \frac{1}{(p-\lfloor p^\lambda\rfloor )H_p}\sum^{p}_{k= \lfloor p^\lambda\rfloor+1}\frac{s_p-s_k}{k}\geq 0,
\end{equation}
where $\lfloor \cdot \rfloor$ denotes the integer part, and $H_p$ is defined in~(\ref{equationdefrieszsummability}).
\end{theo}

\begin{prop}\label{prop.alpha.implies.beta}
 Let $\M$ be a strongly regular sequence, if condition~\eqref{equaCondEquivLogmpOverLogpOmega} holds then also~\eqref{equaCondEquivBetaOmega} is satisfied (and so they are equivalent, by Remark~\ref{remaRieszMethod}).
\end{prop}

\begin{proof}
Since $\M$ is (lc) and (mg), we know by \eqref{eqmg_mppMp} that there exists $a>1$ such that
\begin{equation}\label{equaCotasBetap}
1\leq \frac{m_p}{M^{1/p}_p}=e^{\b_p}\leq a,\qquad p\in\N,
\end{equation}
so $(\b_p)_{p\in\N_0}$ is a bounded sequence of nonnegative numbers. Using~(\ref{eqalfabeta}) we have
  $$\sum^{p-1}_{k=0}\frac{\b_k}{k+1}=\a_p-\b_p$$
  for every $p\in\N$, and then \eqref{equaCondEquivLogmpOverLogpOmega} implies that
  $$\lim_{p\ri\oo} \frac{1}{\log(p)} \sum^{p-1}_{k=0}\frac{\b_k}{k+1}=\lim_{p\ri\oo} \frac{\a_p}{\log (p)}-\frac{\b_p}{\log (p)} = \o(\M). $$
  So, $(\b_p)_{p\in\N_0}$ is Riesz summable with $(L,1)$-sum $\o(\M)$.

Now let us show that $(\b_p)_{p\in\N_0}$ verifies the conditions in
  Theorem~\ref{theoMoricz}. For convenience we define $s_p:=\b_{p-1}$, $p\in\N$.
  It is clear that $(s_p)_{p\in\N}$ is also Riesz summable with sum $\o(\M)$.
Take $\lambda>1$. We apply \eqref{equaCotasBetap} in order to obtain that
  \begin{align}\label{desisumamediasdes}
   \sum_{k=p+1}^{\lfloor p^\lambda\rfloor}\frac{s_k-s_p}{k} \geq -s_p \sum_{k=p+1}^{\lfloor p^\lambda\rfloor}\frac{1}{k} \geq (H_{\lfloor p^\lambda\rfloor}-H_{p})(\log(1/a)).
  \end{align}
We recall that $H_k=\log(k)+\gamma+\ep_k$ with $\lim_{k\ri\oo} \ep_k=0$, where $\ga$ is Euler's constant. Then
\begin{equation*}
 H_{ \lfloor p^\lambda \rfloor}-H_{p}=\log( \lfloor p^\lambda \rfloor)+\gamma+\ep_{ \lfloor p^\lambda \rfloor }-\log(p)-\gamma+\ep_{p},\quad p\in\N.
\end{equation*}
Since for $p\in\N$ we have $p^\lambda-1 \leq \lfloor p^\lambda \rfloor \leq p^\lambda$, also
  $$\log\left(\frac{p^\lambda-1}{p}\right)\leq \log\left(\frac{ \lfloor p^\lambda \rfloor}{p}\right)\leq \log\left(\frac{p^\lambda}{p}\right)= (\lambda-1 )\log(p).$$
Moreover,
  $$\log\left(\frac{p^\lambda-1}{p}\right)=\log \left(p^{\lambda-1}-\frac{1}{p}\right)= (\lambda-1)\log(p)+\log\left(1-\frac{1}{p^\lambda}\right),   $$
  so it is clear that
  $$1+ \frac{\log (1-1/p^\lambda)}{(\lambda-1)\log(p)}\leq \frac{\log\left( \lfloor p^\lambda \rfloor/ p \right)}{(\lambda-1)\log(p)}\leq 1,$$
what allows us to write
\begin{equation}\label{eqdifsumasarmonicas}
 H_{ \lfloor p^\lambda \rfloor}-H_{p}\sim (\lambda-1)\log(p),\quad p\to\infty.
\end{equation}
We also observe that
 $$p^\lambda-1-p\leq \lfloor p^\lambda \rfloor-p\leq p^\lambda-p,$$
 so $ \lfloor p^\lambda \rfloor-p\sim p^\lambda-p \sim p^\lambda$ as $p\ri\oo$. Using this fact,~(\ref{desisumamediasdes}) and~(\ref{eqdifsumasarmonicas}),
 we have
 $$\liminf_{p\ri\oo} \frac{1}{(\lfloor p^\lambda \rfloor-p)H_p}\sum_{k=p+1}^{\lfloor p^\lambda \rfloor}\frac{s_k-s_p}{k}
 \geq\liminf_{p\ri\oo} \frac{(\lambda-1)\log(p)\log(1/a)}{(p^\lambda)\log(p)}=0, $$
 hence
  $$\limsup_{\lambda\ri1^+}\liminf_{p\ri\oo} \frac{1}{(\lfloor p^\lambda \rfloor-p)H_p}\sum_{k=p+1}^{\lfloor p^\lambda \rfloor}\frac{s_k-s_p}{k}
  \geq \limsup_{\lambda\ri1^+} \liminf_{p\ri\oo} \frac{(\lambda-1)\log(p)\log(1/a)}{(p^\lambda)\log(p)} =0,$$
  so~(\ref{eqmoricz1}) holds.

  Let us see that~(\ref{eqmoricz2}) is also true. If we take $0<\lambda<1$, we have
   \begin{align}
   \sum_{k=\lfloor p^\lambda \rfloor+1}^{p}\frac{s_p-s_k}{k} \geq \sum_{k=\lfloor p^\lambda \rfloor+1}^{p}\frac{-\log(a)}{k} =\log(1/a)(H_{p}-H_{\lfloor p^\lambda\rfloor}) ,\no
  \end{align}
  because $0\leq s_k\leq \log(a)$ for every $k\in\N$. As we did before
  $$H_p-H_{\lfloor p^\lambda\rfloor}=\log\left(\frac{p}{\lfloor p^\lambda \rfloor}\right)+\ep_p-\ep_{\lfloor p^\lambda\rfloor}\sim (1-\lambda)\log(p), $$
 $$p-\lfloor p^\lambda\rfloor\sim p-p^\lambda\sim p, $$
 as $p\ri\oo$. So we deduce that
  $$\limsup_{\lambda\ri1^-}\liminf_{p\ri\oo} \frac{1}{(p-\lfloor p^\lambda\rfloor)H_p}\sum_{k=\lfloor p^\lambda\rfloor+1}^{p}\frac{s_p-s_k}{k}
  \geq \limsup_{\lambda\ri1^-} \liminf_{p\ri\oo} \frac{(1-\lambda)\log(p)\log(1/a)}{(p)\log(p)} =0.$$
 Consequently the sequence $(s_k)_{k\in\N}$ verifies the hypotheses of Theorem~\ref{theoMoricz}, and we deduce that
$$\lim_{p\ri\oo} \b_{p-1}=\lim_{p\ri\oo} s_p =\o(\M),$$
as desired.
\end{proof}

We present the main statement of this paper, which is straightforward from Theorem~\ref{teorcondicion3caracterizacion}, Proposition~\ref{prop.alpha.implies.beta} and Proposition~\ref{propcaracdderordenaprox}.

 \begin{theo}\label{theorem.condition3.implies.beta.converge}
    Let $\M$ be a strongly regular sequence, then the following are equivalent:
\begin{enumerate}[(a)]
 \item $d_{\M}(t)$ is a proximate order,
 \item $\lim_{t\ri\oo}d_{\M}(t)=1/\o(\M)$,
 \item $\lim_{p\ri\oo} \log(m_p)/\log(p)=\o(\M)$,
 \item  $\lim_{p\ri\oo} \log\big(m_p/M_p^{1/p}\big)=\o(\M)$.
\end{enumerate}
\end{theo}

\begin{rema}
It turns out that we only need to prove condition (\ref{OA3:1}) (in other words, (b)) to ensure that $d_{\M}$ is a proximate order, (\ref{OA4:1}) being then automatically satisfied. Moreover, given $\M$ strongly regular it is plain to check whether it satisfies (c) or not; indeed, condition (c) holds for every example of strongly regular sequence we have constructed, but its general validity is an open problem.  Moreover, this condition is stable under equivalence.
\end{rema}

\begin{rema}
The second author~\cite[Remark\ 4.11(iii)]{SanzFlat} observed that, for the construction of flat functions in optimal sectors, $d_{\M}$ need not be a proximate order, but it is enough that there exist a proximate order $\ro(t)$ and constants $A,B>0$ such that eventually $A\le (d_{\M}(t)-\ro(t))\log(t)\le B$. Now we know that, in case these inequalities hold, $d_{\M}$ verifies (b) and, consequently, it is indeed a proximate order, what makes that remark superfluous.
\end{rema}

\begin{rema}
Before knowing Theorem~\ref{theorem.condition3.implies.beta.converge}, no easy condition (such as (c)) was known to be equivalent to $d_{\M}$ being a proximate order.
In \cite[Corollary\ 4.10]{SanzFlat} a sufficient condition was given, namely the existence of
  \begin{equation}
   \lim_{p\ri\oo} p\log\left(\frac{m_{p+1}}{m_{p}}\right). \label{corolarioflatfunctionfalso}
  \end{equation}
In the following example we show this condition is not necessary.
\end{rema}

\begin{exam}
We consider the sequence defined by $m_0=m_1=1$, and
$$
m_{2p}=e^{1/p}m_{2p-1},\qquad m_{2p+1}=e^{1/(2p+1)}m_{2p},\quad p\in\N.
$$
Let see that the corresponding sequence $\M$ is strongly regular.
The sequence $(m_p)_{p\in\N_0}$ is clearly (lc).
  Using the results of H.-J. Petzsche and D. Vogt (Proposition~\ref{propPropiedlcmg}(ii.2) and Proposition~\ref{petzche11}), in order to prove (mg) and (snq) it  is enough to see that
  $$1<\inf_{p\in\N}\frac{m_{2p}}{m_p}\leq\sup_{p\in\N}\frac{m_{2p}}{m_{p}}<\oo,$$
what may be done thanks to the well known behaviour of the partial sums of the harmonic series. Moreover, it is plain to check that
$$
\lim_{p\ri\oo} \frac{\log(m_p)}{\log(p)}=\frac{3}{2},
$$
and so $\o(\M)=3/2$ and $d_{\M}$ is a proximate order. However,
  \begin{align}
   &\lim_{p\ri\oo} 2p\log\left(\frac{m_{2p+1}}{m_{2p}}\right)=\lim_{p\ri\oo} 2p\log\left( e^{1/2p+1}\right)= 1 ,  \no\\
   &\lim_{p\ri\oo} (2p-1)\log\left(\frac{m_{2p}}{m_{2p-1}}\right)=\lim_{p\ri\oo} (2p-1)\log\left( e^{1/p}\right)= 2, \no
  \end{align}
and~(\ref{corolarioflatfunctionfalso}) does not hold.
 \end{exam}

\section{Strongly regular sequences and regular variation}

In this section we will recall the notion of regular variation introduced in 1930 by J.~Karamata~(\cite{karamata1,karamata2}), although previous, partial treatments
may be found in the works of E. Landau\cite{landau} and G.~Valiron\cite{valiron13}.
The proof of most of the forthcoming results can be found in the books of
E. Seneta\cite{seneta}  and N. H. Bingham, C. M. Goldie, and J. L. Teugels~\cite{bingGoldTeug}. We will use the theory of regular variation to give a characterization of strongly regular sequences, and we will also show that the growth index $\gamma(\M)$ defined and studied by V. Thilliez~\cite{thilliez}
and the order of quasianalyticity  $\o(\M)$ 
are the same whenever $d_{\M}$ is a proximate order.

\begin{den}
 A measurable function $F:(a,\oo)\ri(0,\oo)$ (with $a>0$) is \textit{regularly varying} if
 \begin{equation}\label{equation.def.RV}
 \lim_{x\ri\oo} \frac{F(\lambda x)}{F(x)} = f(\lambda)\in(0,\oo)
 \end{equation}
for every $\lambda>0$.
\end{den}

\begin{rema} The Uniform Convergence Theorem~\cite[Chap.\ 1.2]{bingGoldTeug}, given by J. Karamata~\cite{karamata1} in the continuous version
and by J. Korevaar and others~\cite{korevaar} in the measurable case,
ensures that the limit in~(\ref{equation.def.RV}) is uniform in every compact set of $(0,\oo)$.
Moreover, the Characterization Theorem \cite[Chap.\ 1.4]{bingGoldTeug} shows that if $F$ is regularly varying, then $f(\lambda)$
in~(\ref{equation.def.RV}) is necessarily of the form $\lambda^\ro$ for some $-\oo<\ro<\oo$ and for each $\lambda>0$.
The number $\ro$ is called the \textit{ index (of regular variation)} of $F$.
\end{rema}

There is an immediate relation between regular variation and proximate orders in our context: $d_{\M}$ is a proximate order if, and only if, the function $M(t)$ is regularly varying (see~\cite[Subsection\ 7.4.1]{bingGoldTeug}).

\begin{den}[\cite{BojanicSeneta}]
A sequence $(s_p)_{p\in\N_0}$ of positive numbers is \textit{regularly varying} if
\begin{equation}\label{equation.def.SRV}
\lim_{p\ri\oo} \frac{ s_{\lfloor \lambda p\rfloor}}{s_p}=\psi(\lambda) \in (0,\oo)
\end{equation}
 for every $\lambda>0.$
\end{den}

\begin{rema}\label{remaEquivRegurVariationFunctionSequence}
As it can be found in~\cite[Th.\ 2]{BojanicSeneta}, a sequence $(s_p)_{p\in\N_0}$ is regularly varying if, and only if, the function $F(x)=s_{\lfloor x\rfloor}$, $x>0$, is regularly varying.
\end{rema}

The following theorem of L.~de~Haan~\cite{haan} shows that if we have monotonicity
we can restrict the limit in~(\ref{equation.def.SRV}) to the integers values of $\lambda$. In fact, we only need to prove~(\ref{equation.def.SRV}) for two suitable integer values of $\lambda$.

\begin{theo}[\cite{haan},\ Th.\ 1.1.2] A positive monotone sequence $(s_p)_{p\in\N_0}$ varies regularly if there exist
positive integers $\ell_1,\ell_2$ with $\log(\ell_1)/\log(\ell_2)$ finite and irrational such that for some real number $\ro$,
$$\lim_{p\ri\oo} \frac{s_{\ell_j p}}{s_p}=\ell_j^{\ro}, \qquad j=1,2.$$
\end{theo}

\begin{prop}\label{pro.SRS.implica.RV}
 Let $(M_p)_{p\in\N_0}$ a strongly regular sequence such that
 $$\lim_{p\ri\oo} \frac{\log(m_p)}{\log(p)}=\o(\M).$$
Then $\m$ is regularly varying, and
 $$\lim_{p\ri\oo} \frac{m_{\ell p}}{m_p}=\ell^{\o(\M)}, \qquad \ell\geq 2.$$
\end{prop}

\begin{proof}
 If we consider the sequences $(\a_p)_{p\in\N_0}$ and $(\b_p)_{p\in\N_0}$ defined in Notation~\ref{notacionAlfaBeta}, we have to show that
 $$\lim_{p\ri\oo}( \a_{\ell p}-\a_p)=\o(\M)\log(\ell),\qquad \ell\geq2.$$
For simplicity we note $\o=\o(\M)$. Using~(\ref{eqalfabeta}) we see that
\begin{equation}
 \a_{\ell p}-\a_p=\sum^{\ell p-1}_{k=p}\frac{\b_k}{k+1}+\b_{\ell p}-\b_p, \qquad p,\ell\geq 2.\no
\end{equation}
By Proposition~\ref{prop.alpha.implies.beta} we have
$\lim_{p\ri\oo}\b_p=\o$,
so it is sufficient to prove that
$$\lim_{p\ri\oo} \sum^{\ell p-1}_{k=p}\frac{\b_k}{k+1} =\o(\M)\log(\ell).$$
If we take $\ep>0$, we fix $\delta>0$ such that $\delta\log(\ell)<\ep/6$. 
There exists $p_\delta\in\N$
such that
$|\b_p-\o|<\delta$ for $p\geq p_\delta$.
In the notation of~(\ref{equationdefrieszsummability}), we remember that $H_p=\log(p)+\ga+\ep_p$ with $\lim_{p\ri\oo} \ep_p=0$, consequently for $p\geq p_\delta$ we have
$$\sum^{\ell p-1}_{k=p}\frac{\b_k}{k+1} \leq (\o+\delta) (H_{\ell p}-H_p) = (\o+\delta) (\log(\ell) +\ep_{\ell p}-\ep_p).$$
Now take $p_0\geq p_\delta$ such that for $p\geq p_0$ one has
$$|\o\ep_p|<\ep/6,\quad |\delta\ep_p|<\ep/6,$$
then for $p\geq p_0$ we see that
$$\sum^{\ell p-1}_{k=p}\frac{\b_k}{k+1} \leq \o\log(\ell)+\o\ep_{\ell p} -\o\ep_p + \delta \log(\ell) +\delta\ep_{\ell p}-\delta\ep_p<\o\log(\ell)+\ep.$$
Analogously, for $p\geq p_0$ we see that
$$ \o\log(\ell)-\ep<\sum^{\ell p-1}_{k=p}\frac{\b_k}{k+1},$$
and we are done.
 \end{proof}

\begin{rema}
 For a given logarithmically convex sequence $\M$, applying Proposition~\ref{petzche11} to $(p!M_p)_{p_\in\N_0}$, which is also logarithmically
 convex, we see that (snq) and (\ref{condition.liminf.mkp.over.mp.great.1})
 are equivalent conditions. If we use
 Proposition~\ref{propPropiedlcmg}(i) and (ii.2),
 we can establish the following result.
\end{rema}

\begin{prop}\label{alt.def.SRS} Let  $(M_p)_{p\in\N_0}$ be a sequence of positive numbers with $M_0=1$. The following are equivalent:
\begin{itemize}
\item[(i)] $\M$ is strongly regular,
\item[(ii)] $\bm$ is non-decreasing, $\displaystyle\sup_{p\in\N} \frac{m_{2p}}{m_{p}}<\oo$ and there exists $k\in\N$ such that $\displaystyle\liminf_{p\ri\oo} \frac{m_{kp}}{m_p}>1$.
\end{itemize}
\end{prop}

In order to characterize strongly regular sequences satisfying \eqref{limit.logmp.over.logp.omega}, we will apply the following result of R. Bojanic and E. Seneta.

\begin{theo}[\cite{BojanicSeneta},\ Th. 3]\label{theoBojanicSeneta}
If $(s_p)_{p\in\N}$ is a regularly varying sequence of index $\o$, then there exist sequences of positive numbers
$(C_p)_{p\in\N}$ and $(\delta_p)_{p\in\N}$ converging to $C\in(0,\oo)$ and zero,
respectively, such that
$$s_p=p^{\o} C_p \exp  \left(\sum^{p}_{j=1}\delta_j/j\right),\quad p\in\N. $$
\end{theo}

\begin{theo} \label{teocaractsucfrbyregvar}(Characterization Theorem)
Let  $(M_p)_{p\in\N_0}$ be a sequence of positive numbers with $M_0=1$. The following are equivalent:
\begin{itemize}
\item[(i)] $\M$ is strongly regular and $\lim_{p\ri\oo} \log(m_p)/\log(p)=\o(\M)$,
\item[(ii)] $\M$ is (lc) and $\bm$ is regularly varying of index $\o(\M)$, i.e.,
 $$\lim_{p\ri\oo} \frac{m_{\ell p}}{m_p}=\ell^{\o(\M)}, \qquad \ell\geq 2.$$
\end{itemize}
\end{theo}

\begin{proof}
Proposition~\ref{pro.SRS.implica.RV} shows that (i) implies (ii).\par
For the converse, since
 $$\lim_{p\ri\oo}\frac{m_{2p}}{m_p}=2^{\o(\M)}>1,$$
it is clear that all the conditions in Proposition~\ref{alt.def.SRS}(ii) are satisfied, and so $\M$ is strongly regular. Finally, using
Theorem~\ref{theoBojanicSeneta} we know that
$$m_p=p^{\o(\M)} C_p \exp  \left(\sum^{p}_{k=1}\delta_k/k\right),\quad p\in\N,$$
where $C_p\ri C$ and $\delta_p\ri0$ as $p\ri\oo$. Consequently,
$$\lim_{p\ri\oo} \frac{\log(m_p)}{\log(p)}= \lim_{p\ri\oo}\Big( \o(\M)
+ \frac{\log (C_p)}{\log(p)}+ \frac{1}{\log(p)} \sum^{p}_{k=1} \frac{\delta_k}{k}\Big) =\o(\M), $$
since $\lim_{p\ri\oo} (\log(p))^{-1} \sum^{p}_{k=1} \delta_k/k =0$.
\end{proof}

We now turn to the last aim in this paper, concerning the equality of $\o(\M)$ and the growth index $\ga(\M)$ introduced by V. Thilliez~\cite{thilliez}. It is necessary to recall the connection between regular variation and the notion of almost increasing functions or sequences.

\begin{den}
A function $f$ positive and finite in $[b,\oo)$ is said to be \textit{almost increasing} in $[c,\oo)$ ($c\geq b$) if there
exists a constant $M\geq1$ such that
\begin{equation}
 f(x)\leq M f(y), \qquad \text{for every} \quad y\geq x \geq c,\no
\end{equation}
or, equivalently, if
\begin{equation*}
 f(x)\leq M \inf_{y\geq x} f(y), \qquad \text{for each} \quad x \geq c.
\end{equation*}
A sequence $(s_p)_{p\in\N_0}$ is \textit{almost increasing} if
\begin{equation}
s_p\leq M \inf_{\ell\geq p} s_{\ell}, \qquad \text{for every} \quad p \in \N_0.\no
\end{equation}
\end{den}

The next proposition is due to J. Karamata~\cite{karamata1,karamata2} for continuous functions, but it also holds for measurable ones.

\begin{pro}[\cite{seneta},\ Section\ 1.5]\label{propKaramataRV}
Let $F$ be a regularly varying function $F:(a,\oo)\ri(0,\oo)$ (where $a>0$) with index $\ro$. Then, for every $\sigma<\ro$ one has
$$\lim_{x\ri\oo} \frac{\inf_{y\geq x} (y^{-\sigma} F(y))}{ x^{-\sigma} F(x)} =1,
\qquad \lim_{x\ri\oo} \frac{\sup_{a\leq y\leq x} (y^{-\sigma} F(y))}{ x^{-\sigma} F(x)} =1,  $$
and for every $\ro<\tau$ one has
$$\lim_{x\ri\oo} \frac{\sup_{y\geq x} (y^{-\tau} F(y))}{ x^{-\tau} F(x)} =1,
\qquad \lim_{x\ri\oo} \frac{\inf_{a\leq y\leq x} (y^{-\tau} F(y))}{ x^{-\tau} F(x)} =1. $$
\end{pro}

The following proposition shows the aforementioned connection. We include the proof for the sake of completeness.

\begin{pro}\label{pro.RV.implica.almostincreasing}
Let $F$ be a regularly varying function $F:[a,\oo)\ri(0,\oo)$ (where $a>0$), with index $\ro$. We also suppose that $F$ is non-decreasing and
locally bounded. Then, for each $\sigma<\ro$ the function
$x^{-\sigma}F(x)$ is almost increasing in $[a,\oo)$.
\end{pro}

\begin{proof}
 We fix $\sigma<\ro$. From the first limit in Proposition~\ref{propKaramataRV},
given $\ep\in(0,1)$ there exists $x_0>a$ such that
$${ x^{-\sigma} F(x)}\leq \frac{1}{1-\ep} \inf_{y\geq x} (y^{-\sigma} F(y)), \qquad x\geq x_0.  $$
We define $M_1:=1/(1-\ep)>1$. On the other hand, since $F$ is increasing and locally bounded, there exist $c,C>0$ such that $F(x)\in[c,C]$ for every $x\in[a,x_0]$. Consequently,
$$ \frac{y^\sigma F(x)}{F(y) x^\sigma } \leq \frac{(x_0)^\sigma C}{c a^\sigma}=:M_2,\qquad y,x\in[a,x_0].$$
We consider $M:=M_1M_2$, then whenever $x\in[a,x_0]$ we have that
$$
 \frac{F(x)}{x^\sigma}\leq M_2 \frac{F(y)}{y^\sigma} \leq M \frac{F(y)}{y^\sigma}, \qquad y\in[x,x_0],
 $$
 and
 $$
  \frac{F(x)}{x^\sigma}\leq M_2 \frac{F(x_0)}{x_0^\sigma} \leq M \inf_{y\geq x_0} \frac{F(y)}{y^\sigma},
  $$
so we deduce that $F$ is almost increasing in $[a,\oo)$
\end{proof}

We recall the definition of the growth index $\gamma(\M)$ given  by V. Thilliez

\begin{defi}
Let $\bM=(M_{p})_{p\in\N_{0}}$ be a strongly regular sequence and $\ga>0$. We say $\bM$ satisfies property
$\left(P_{\ga}\right)$  if there exist a sequence of real numbers $m'=(m'_{p})_{p\in\N_0}$ and a
constant $a\ge1$ such that: (i) $a^{-1}m_{p}\le m'_{p}\le am_{p}$, $p\in\N$, and (ii) $\left((p+1)^{-\ga}m'_{p}\right)_{p\in\N_0}$
is increasing.

The \textit{growth index} of $\bM$ is
$$\ga(\bM):=\sup\{\ga\in\R:(P_{\ga})\hbox{ is fulfilled}\}\in(0,\infty).$$
\end{defi}

\begin{prop}\label{pro.pgamma.almostincreasing}
Let $\bM=(M_{p})_{p\in\N_{0}}$ be a strongly regular sequence and $\ga>0$. Then, $\bM$ satisfies property
$\left(P_{\ga}\right)$  if, and only if, the sequence $\left((p+1)^{-\ga}m_{p}\right)_{p\in\N_0}$ is almost increasing.
\end{prop}

\begin{proof}
If $\M$ satisfies $\left(P_{\ga}\right)$ with $m'=(m'_{p})_{p\in\N_0}$ and a
constant $a\ge1$, and we take $\ell,p\in\N_0$ with $\ell\geq p$, we have
$$(p+1)^{-\ga}m_{p}\leq a (p+1)^{-\ga} m'_p \leq a (\ell+1)^{-\ga} m'_{\ell} \leq a^2 (\ell+1)^{-\ga} m_{\ell},$$
so $\left((p+1)^{-\ga}m_{p}\right)_{p\in\N_0}$ is almost increasing.\par
\noindent Conversely, if
 $\left((p+1)^{-\ga}m_{p}\right)_{p\in\N_0}$ is almost increasing, there exists $a\geq1$ such that
  $$(p+1)^{-\ga}m_{p} \leq a (\ell+1)^{-\ga}m_{\ell},\qquad \ell\geq p.  $$
 We define $m'_p:=(p+1)^{\ga} \inf_{\ell\geq p}  (\ell+1)^{-\ga} m_{\ell}$. We have:
 \begin{enumerate}[(i)]
  \item For every $p\in\N$,
  $$a^{-1}m_{p} \le (p+1)^{\ga} \inf_{\ell\geq p} (\ell+1)^{-\ga} m_{\ell}=  m'_p \le (p+1)^{\ga} (p+1)^{-\ga} m_p \le  a m_p.$$
  \item For every $\ell,p\in\N$ with  $\ell\geq p$,
  $$(p+1)^{-\ga}m'_{p}= \inf_{q\geq p}[ (q+1)^{-\ga} m_q] \le (\ell+1)^{-\ga} (\ell+1)^{\ga}  \inf_{q\geq \ell} [(q+1)^{-\ga} m_q ] =(\ell+1)^{-\ga}m'_{\ell}.$$
 \end{enumerate}
Then, $\M$ satisfies  $\left(P_{\ga}\right)$.
 \end{proof}

\begin{rema}\label{remaGammaMAlmostIncreasing}
The result above shows that the growth index can also be defined as
$$\ga(\bM):=\sup\{\ga\in\R: \hbox{the sequence} \left((p+1)^{-\ga}m_{p}\right)_{p\in\N_0} \hbox{ is almost increasing}\}.$$
\end{rema}

The following theorem was proved by V. Thilliez \cite{thilliez} and, subsequently, by A. Lastra and the second author \cite{lastrasanz1}.

\begin{theo}[\cite{thilliez}]
Let $0<\ga<\ga(\M)$.
Then, the class $\mathcal{\tilde{A}}_{\M}(S_{\gamma})$ is not quasianalytic.
\end{theo}

From this result, and by the very definition of $\omega(\M)$, the following holds.

\begin{prop}[\cite{SanzFlat},\ Prop.\ 3.7]\label{pro.gamma.menor.omega}
For any strongly regular sequence $\M$ one has $\omega(\M)\ge \gamma(\M)$.
\end{prop}

We are ready to prove our last statement.

\begin{theo}
 Let $\M$ be a strongly regular sequence such that $\lim_{p\ri\oo} \log(m_p)/\log(p)=\o(\M)$, then $\o(\M)=\gamma(\M)$.
\end{theo}

\begin{proof}
According to Proposition~\ref{pro.SRS.implica.RV}, $\bm$ is regularly varying, and so is the function $f(x)=m_{\lfloor x\rfloor}$, $x>0$, by Remark~\ref{remaEquivRegurVariationFunctionSequence}. Since $f$ is also non-decreasing and locally bounded, by Proposition~\ref{pro.RV.implica.almostincreasing} the function $x^{-\ga} m_{\lfloor x \rfloor}$ is almost increasing for every $\ga<\o(\M)$.
It follows that the sequence $\left((p+1)^{-\ga}m_{p}\right)_{p\in\N_0}$ is almost increasing, and by
Remark~\ref{remaGammaMAlmostIncreasing} we conclude that $\o(\M)\leq\ga(\M)$. Proposition~\ref{pro.gamma.menor.omega} leads to the conclusion.
\end{proof}

\noindent\textbf{Acknowledgements}: Both authors are partially supported by the Spanish Ministry of Economy and Competitiveness under project MTM2012-31439. The first author is partially supported by the University of Valladolid through a Predoctoral Fellowship (2013 call) co-sponsored by the Banco de Santander.

\vskip.5cm
\noindent Authors' Affiliation:\par\vskip.5cm
Departamento de \'Algebra, An\'alisis Matem\'atico, Geometr{\'\i}a y Topolog{\'\i}a\par
Instituto de Investigaci\'on en Matem\'aticas de la Universidad de Valladolid, IMUVA\par
Facultad de Ciencias\par
Universidad de Valladolid\par
47011 Valladolid, Spain\par
E-mail: jjjimenez@am.uva.es (Javier Jim\'enez-Garrido), jsanzg@am.uva.es (Javier Sanz).
\par\vskip.5cm

\end{document}